\documentclass[preprint]{imsart}

\RequirePackage[OT1]{fontenc}
\RequirePackage{amsthm,amsmath,amssymb}
\RequirePackage[numbers]{natbib}
\RequirePackage[colorlinks,citecolor=blue,urlcolor=blue]{hyperref}

\usepackage{nicefrac}
\usepackage{bbm}

\newcommand{\di}{\, \mathrm{d}}

\newcommand{\R}{\mathbb{R}}
\newcommand{\cK}{\mathcal{K}}

\newcommand{\N}{\mathbb{N}}

\newcommand{\CLe}{C_{\text{Lep}}}

\newcommand{\ind}{\mathbbm{1}}

\newtheorem*{definition}{Definition}

\newcommand{\sumdel}{\sum_{j,n,\delta}\,}

\DeclareMathOperator{\E}{\mathbb{E}}
	\DeclareMathOperator{\Pb}{\mathbb{P}}

\DeclareMathOperator{\essup}{essup}

\arxiv{arXiv:1902.05261}

\startlocaldefs
\numberwithin{equation}{section}
\theoremstyle{plain}
\newtheorem{thm}{Theorem}[section]
\newtheorem{prop}{Proposition}[section]
\newtheorem{lem}{Lemma}[section]
\newtheorem{ass}{Assumption}

\endlocaldefs

\begin{document}

\begin{frontmatter}
\title{Rate-optimal nonparametric estimation for random coefficient regression models}
\runtitle{Random Coefficients}

\begin{aug}
\author{\fnms{Hajo} \snm{Holzmann}\ead[label=e1]{holzmann@mathematik.uni-marburg.de}}
\and
\author{\fnms{Alexander} \snm{Meister}
\ead[label=e2]{alexander.meister@uni-rostock.de}}

\runauthor{H. Holzmann and A. Meister}

\affiliation{Philipps-Universit\"at Marburg\thanksmark{m1}}
\affiliation{Universit\"at Rostock\thanksmark{m2}}

\address{Hajo Holzmann \\ Fachbereich Mathematik und Informatik, \\ Philipps-Universit\"at Marburg,\\ 35037 Marburg, Germany. \\
\printead{e1}}

\address{Alexander Meister \\ Institut f\"ur Mathematik, \\ Universit\"at Rostock, \\ 18051 Rostock, Germany. \\
\printead{e2}}

\end{aug}

\begin{abstract}
Random coefficient regression models are a popular tool for analyzing unobserved heterogeneity, and have seen renewed interest in the recent econometric literature. In this paper we obtain the optimal pointwise convergence rate for estimating the density in the linear random coefficient model over H\"older smoothness classes, and in particular show how the tail behavior of the design density impacts this rate. In contrast to previous suggestions, the estimator that we propose and that achieves the optimal convergence rate does not require dividing by a nonparametric density estimate. The optimal choice of the tuning parameters in the estimator depends on the tail parameter of the design density and on the smoothness level of the H\"older class, and we also study adaptive estimation with respect to both parameters.
\end{abstract}



\end{frontmatter}

\section{Introduction}

In this paper we consider the linear random coefficient regression model, in which i.i.d.~(independent and identically distributed) data $(X_j,Y_j)$, $j=1,\ldots,n$ are observed according to 
\begin{equation} \label{eq:model}
Y_j \, = \, A_{0,j} + A_{1,j} X_j\,.
\end{equation}
Therein $A_j := (A_{0,j},A_{1,j})$ are unobserved i.i.d.~random variables with the bivariate Lebesgue density $f_A$; while $A_j$ and $X_j$ are independent. Note that (\ref{eq:model}) represents a randomized extension of the standard linear regression model.  
We shall derive the optimal convergence rates for estimating $f_A$ over H\"older smoothness classes in case when the $X_j$ have a Lebesgue density $f_X$ with polynomial tail behaviour, as specified in Assumption \ref{assump::designdens} below.

From a parametric point of view with focus on means and variances of the random coefficients, a multivariate version of model (\ref{eq:model}) is studied by \cite{HH68}. They assume the coefficients $A_j$ to be mutually independent. The nonparametric analysis of model (\ref{eq:model}) has been initiated by  \cite{BH92} and \cite{BM94}. \cite{BFH96} use Fourier methods to construct an estimator of $f_A$. They do not derive the optimal convergence rate, though. Furthermore, their estimator is rather involved as it requires a nonparametric estimator of a conditional characteristic function, which is then plugged into a regularized Fourier inversion.   

Extensions of model (\ref{eq:model}) have seen renewed interest in the econometrics literature in recent years. \cite{HKM10} suggest a nonparametric estimator in a multivariate version of model (\ref{eq:model}). They only obtain its convergence rate for very heavy tailed regressors. Moreover, their estimator requires dividing by a nonparametric density estimator for a transformed version of the regressors. This involves an additional smoothing step, and potentially renders the estimator unstable. \cite{BH18} propose a specification test for model (\ref{eq:model}) against a general nonseparable model as the alternative, while \cite{DEPSH17} suggest multiscale tests for qualitative hypotheses on $f_A$. 
Extensions and modifications of model (\ref{eq:model}) are studied in \cite{GK13}, \cite{LP17}, \cite{AB11}, \cite{GH11}, \cite{GL18}, \cite{M17}, \cite{MT13} and \cite{HHM17}. Methods of analytic continuation of the coefficients density outside the support of the covariates are considered  under more restrictive conditions in \cite{HHM17} and in the recent work of \cite{GG19}.

In this paper, we consider the basic model (\ref{eq:model}) under the following condition.
\begin{ass}[Design density]\label{assump::designdens}
	For some constants $\beta>0$ and $C_X > c_X > 0$, the density $f_X$ satisfies
	\begin{equation} \label{eq:decayX}
	C_X (1 + |x|)^{-\beta-2} \geq f_X(x) \geq c_X\cdot (1 + |x|)^{-\beta-2}\,, \qquad \forall x \in \mathbb{R}\,,
	\end{equation}
\end{ass}
We analyze precisely how the tail parameter $\beta$ of $f_X$ influences the optimal rate of convergence of $f_A$ at a given point $a \in \R^2$ in a minimax sense in case $\beta >1$ . 
Note that the heavy tailed setting which is studied in \cite{HKM10} corresponds to $\beta = 0$ in Assumption \ref{assump::designdens}. To our best knowledge a rigorous study of the minimax convergence rate in the more realistic case of $\beta>1$ has been missing so far. Indeed we fill this gap and derive optimal rates, which are fundamentally new and not known from any other nonparametric estimation problem. 

The estimator which we propose is inspired by \cite{HHM17}. It achieves the optimal convergence rate and does not require dividing by a nonparametric density estimator. Instead we exploit the order statistic of the transformed design variables in a Priestley-Chao manner. The optimal choice of the tuning parameters depends both on the two parameters $\beta$ and on the smoothness parameter of the H\"older class, which is reminiscent of the estimation problem in \cite{JMR14} and in contrast to usual adaptation problems in nonparametric curve estimation, in which the smoothing parameters shall adapt only to an unknown smoothness level. Here we show how to make the estimator adaptive with respect to both of these parameters. 

The paper is organized as follows. In Section \ref{sec:estimator} we introduce our estimation procedure. Section \ref{sec:bounds} is devoted to upper and lower risk bounds, which yield minimax rate optimality for the pointwise risk. We also derive an upper risk bound for the uniform risk, here, an additional logarithmic factor occurs. In Section \ref{sec:adapt} we deal with adaptivity. The proofs and technical lemmas are deferred to Section \ref{sec:proofs}.   

Let us fix some notation: $\psi_A$ denotes the characteristic function of the $A_j$, while $\psi_{U|Z}$ is the conditional characteristic function of the random variable $U$ given the random variable $Z$. Throughout $|\cdot|$ stands for the Euclidean norm of a real or complex vector, and $\ind(A)$ denotes the indicator function of the event $A$. For positive sequences $(a_n)$ and $(b_n)$ we write $a_n \asymp b_n$ if $c \, a_n \leq b_n \leq C \, a_n$, $n \in \N$ for constants $0 < c < C$.

\section{The estimator}\label{sec:estimator}

In order to construct an estimator for $f_A$ in model (\ref{eq:model}), 
we transform the data $(X_j,Y_j)$ into $(Z_j,U_j)$ via
\[ U_j = Y_j/\sqrt{1+X_j^2}\,, \qquad \big(\cos Z_j, \sin Z_j \big) = (1,X_j)/\sqrt{1+X_j^2}\,,\]
so that $Z_j \in (- \pi/2, \pi/2)$ almost surely (a.s.), $Z_j$ and $A_j$ are independent, and
\begin{equation}\label{eq:normalizedmodel}
 U_j = A_{0,j}\, \cos Z_j + A_{1,j}\, \sin Z_j\,.
\end{equation}
Then the conditional characteristic function $\psi_{U|Z}(\cdot | z)$ of $U_j$ given $Z_j=z$ equals 
\begin{equation}\label{eq:relcharafct}
 \psi_{U|Z}(t|z) = \psi_A\big(t \cos z, t \sin z \big)\,. 
\end{equation}
By Fourier inversion, integral substitution into polar coordinates (with signed radius) and (\ref{eq:relcharafct}) we deduce that  
\begin{align}\label{eq:identifiedrelation} 
f_A(a) & = \frac{1}{(2 \pi)^2}\, \iint \exp\big( - i a'b \big)\, \psi_A (b)\, \di b \notag \\  
& = \frac{1}{(2 \pi)^2}\, \int_\R\, \int_{- \pi/2}^{\pi/2} |t|\, \exp\big( - i t (a_0 \cos z + a_1 \sin z) \big)\, \psi_{U|Z} \big(t|z\big)\, \di z\, \di t\,.
\end{align}
%
%
The equation (\ref{eq:identifiedrelation}) motivates us to estimate $f_A$ by an empirical version of the conditional characteristic function $\psi_{U|Z}$ which is directly accessible from the data $(Z_j,U_j)$. 
For that purpose choose a function $w$ which satisfies the following assumption.  
\begin{ass}[Kernel]\label{assum:kernel}
	For a number $\ell \in \N_0$ the function $w: \R \to \R$ is even, supported on $[-1,1]$, $(\ell+1)$-fold continuously differentiable on the whole real line, satisfies $w(0)=1$ as well as $w^{(k)}(0) = 0$ for all $k=1,\ldots,\ell$, and $|w|$ is bounded by $1$.	
\end{ass}
Assumption \ref{assum:kernel} could be relaxed somewhat. In particular, we may assume compact support instead of imposing the support of $w$ to be a subset of $[-1,1]$ and we may remove the condition that $|w|$ is bounded by $1$. Simple boundedness is sufficient, which follows from the other conditions.

Now we consider the regularized version of $f_A$ by kernel smoothing as follows
\begin{align} \nonumber 
& \tilde f_A(a; h) \\ \nonumber & = \frac{1}{(2 \pi)^2} \int_\R \int_{- \pi/2}^{\pi/2} w(th) |t| \exp\big( - i t (a_0 \cos z + a_1 \sin z) \big) \psi_{U|Z} \big(t|z\big) \di z \di t \\ \label{eq:regexpress}
& = \int_{- \pi/2}^{\pi/2} \, \int_\R K\big( u - a_0 \cos z - a_1 \sin z;h\big)\, f_{U|Z}(u|z)\, \di u\, \di z\,, 
\end{align}
where 
\begin{align}\label{eq:resulkernel}
K\big(x;h \big):=  \frac{1}{(2 \pi)^2}\, \int_\R\,  w(th)\, |t|\, \exp( i t x)\, \di t
=  \frac{2}{(2 \pi)^2}\, \int_0^\infty\,  w(th)\, t\, \cos(t x)\, \di t.
\end{align}
Inspired by (\ref{eq:regexpress}) we introduce a Priestley-Chao type estimator of the density $f_A$, 
\begin{align} \nonumber
\hat f_A(a; h, \delta) & = \sum_{j=1}^{n-1}\, K\big( U_{[j]} - a_0 \cos Z_{(j)} - a_1 \sin Z_{(j)};h\big)\,\big(Z_{(j+1)} - Z_{(j)} \big) \\ \nonumber & \hspace{5cm} \cdot \ind(-\pi/2 + \delta \leq Z_{(j)} \leq Z_{(j+1)} \leq \pi/2 - \delta)\\  \nonumber 
& = \frac{1}{(2 \pi)^2}\, \int_\R\, w(th)\, |t|\, \sum_{j=1}^{n-1}\, \exp\big(  i t \big(U_{[j]} - a_0 \cos Z_{(j)} - a_1 \sin Z_{(j)}\big) \big)\\ \label{eq:est}
& \hspace{2cm}\cdot \big(Z_{(j+1)} - Z_{(j)} \big)  \,\ind(-\pi/2 + \delta \leq Z_{(j)} \leq Z_{(j+1)} \leq \pi/2 - \delta)\,\di t,
\end{align}
where $(U_{[j]}, Z_{(j)})$, $j=1, \ldots, n$, denotes the sample $(U_{j}, Z_{j})$, $j=1, \ldots, n$, sorted such that $Z_{(1)} \leq \ldots \leq Z_{(n)}$, and where $h = h_n>0$ is a classical bandwidth parameter and $\delta = \delta_n \geq 0$ is a threshold parameter both of which remain to be selected. By the parameter $\delta$ we cut off that subset of the interval $[-\pi/2,\pi/2]$ in which the $Z_j$ are sparse.  

In the following we shall use the symbol
\begin{equation}\label{eq:sumshortnot}
\sumdel := \sum_{j \in \{ 1, \ldots, n\},  -\pi/2 + \delta \leq Z_{(j)} \leq Z_{(j+1)}\leq \pi/2 - \delta}
\end{equation}
to denote the sum over the random set of indices $1 \leq j \leq n-1$ for which $-\pi/2 + \delta \leq Z_{(j)} \leq Z_{(j+1)}\leq \pi/2 - \delta$. Thus, we may write the estimator in \eqref{eq:est} as
\[ \hat f_A(a; h, \delta)  = \sumdel\, K\big( U_{[j]} - a_0 \cos Z_{(j)} - a_1 \sin Z_{(j)};h\big)\,\big(Z_{(j+1)} - Z_{(j)} \big).\]

In this paper we consider one-dimensional covariates $Z_1,\ldots,Z_n$ only. From a methodological point of view, the estimator (\ref{eq:est}) could be extended to the multivariate setting by using Voronoi cells instead of the order statistics. A similar technique is proposed in eq. (36) in \cite{HHM17}. On the other hand, the asymptotic properties of such an estimator might be completely different from the univariate case.

\section{Upper and lower risk bounds}\label{sec:bounds}

We consider the following H\"older smoothness class of densities. 
\begin{definition}

 For a point $a = (a_0,a_1) \in \mathbb{R}^2$, a smoothness index $\alpha>0$ and constants $c_A, c_B, r_A, c_M>0$ define the class ${\cal F} = {\cal F}(a,\alpha,c_A,c_B,r_A,c_M)$ of densities as follows: $f_A \in {\cal F}(a,\alpha,c_A,c_B,r_A,c_M)$ is H\"older-smooth of the degree $\alpha$ in the neighborhood $U_{r_A}(a) = \{b \in \R^2 \mid |a-b| < r_A\}$, that is, $f_A$ is $s= \lfloor \alpha \rfloor = \max\{ k \in \N_0 \mid k < \alpha\}$-times continuously differentiable in $U_{r_A}(a)$ and its partial derivatives satisfy
 {\small
\begin{equation} \label{eq:Hoelder}
\Big|\frac{\partial^s f_A}{\partial x^k \partial y^{s-k}}(x,y) - \frac{\partial^s f_A}{\partial x^k \partial y^{s-k}}(a_0,a_1)\Big| \, \leq \, c_A \cdot \big|(x,y) - a\big|^{\alpha-s}\,,   
\end{equation}
}
for all $k=0,\ldots,s$ and $(x,y)\in U_{r_A}(a)$. Furthermore, assume that the Fourier transform $\psi_A$ of $f_A$ is weakly differentiable and its weak derivative $\nabla \psi_A$ satisfies  
\begin{equation} \label{eq:Lipschitz}
\int \underset{y \in \R}{\essup} \big|\nabla\psi_A(x,y)\big| \di x\,  \leq \, c_B\,,
\end{equation}
and that $f_A(a) \leq c_M$ for all $a \in \mathbb{R}^2$. 
\end{definition}

For the proof of the first theorem, the global partial tail and smoothness condition (\ref{eq:Lipschitz}) of the order $1$ is required in addition to the local smoothness assumption (\ref{eq:Hoelder}) of the order $\alpha$. The theorem provides an upper bound on the convergence rate for the estimator in (\ref{eq:est}). 	

\begin{thm} \label{T:1}
	Consider model (\ref{eq:model}) and assume that $f_X$ satisfies (\ref{eq:decayX}) for some $\beta>1$. If $w$ satisfies Assumption \ref{assum:kernel} for $l \geq 2 \, \lfloor \alpha \rfloor$, and if $\delta = \delta_n$ and $h = h_n$ are chosen such that 
	\[ \delta \asymp n^{-\frac1{\beta+1}}, \quad \text{ and }\quad h \asymp n^{-\frac1{(\alpha+2)(\beta+1)}},\]
	 then the estimator (\ref{eq:est}) attains the following asymptotic risk upper bound over the function class ${\cal F} = {\cal F}(a,\alpha,c_A,c_B,r_A,c_M)$, 
	$$ \sup_{f_A \in {\cal F}} \E_{f_A} \big[\big|\hat{f}_A(a;h,\delta) - f_A(a)\big|^2\big] \, = \, {\cal O}\big(n^{-\frac{2\, \alpha}{(\alpha+2)(\beta+1)}}\big)\,. $$
\end{thm}

The following theorem yields that the convergence rates which our estimator (\ref{eq:est}) achieves according to Theorem \ref{T:1} are optimal for the pointwise risk in the minimax sense.

\begin{thm} \label{T:low}
	Fix $a=0$ and the constants $c_A$, $c_B$ sufficiently large for any $\alpha>0$ and $\beta>1$. Let $(\hat{f}_n)_n$ be an arbitrary sequence of estimators of $f_A$, where $\hat{f}_n$ is based on the data $(X_j,Y_j)$, $j=1,\ldots,n$, for each $n$. Assume that $f_X$ satisfies (\ref{eq:decayX}). Then 
	$$ \liminf_{n\to\infty}\, n^{\frac{2\, \alpha}{(\alpha+2)(\beta+1)}} \, \sup_{f_{A} \in {\cal F}} \E_{f_A} \big[\big|\hat{f}_n(0) - f_A(0)\big|^2\big] \, > \, 0\,. $$
\end{thm}

The convergence rates from Theorem \ref{T:1} and \ref{T:low} differ significantly from standard rates in nonparametric estimation. While they become faster as $\alpha$ increases, they become slower as $\beta$ gets larger. It is remarkable that they do not approach the (squared) parametric rate $n^{-1}$ but the slower rate $n^{-2/(\beta+1)}$ for large $\alpha$.  

\textit{The case $\beta \leq 1$.}\quad 
	An analysis of the proof of Theorem \ref{T:1} shows that in case $\beta < 1$, choosing $\delta \asymp n^{-\frac1{\beta+1}}$ and $h \asymp n^{-\frac1{2\alpha+4}}$ gives the rate 
	$$ \sup_{f_A \in {\cal F}} \E_{f_A} \big[\big|\hat{f}_A(a;h, \delta) - f_A(a)\big|^2\big] \, = 
	{\cal O}\big(n^{-\frac{2 \alpha}{2\alpha+4}}\big); 
	$$
	in case $\beta =1$, an additional logarithmic factor occurs. The upper bound no longer depends on $\beta$ in this regime. For $\beta=0$, \cite{HKM10} obtain the faster rate ${\cal O}\big(n^{-\frac{2 \alpha}{2\alpha+3}}\big)$; their rate is in $\mathcal{L}_2$ but could be transferred to a pointwise rate. However, they additionally impose the assumption that the density $f_A$ is uniformly bounded with a bounded support. This implies that $f_{U|Z}$ is also uniformly bounded. Under this additional assumption, instead of (\ref{eq:variancebound}) in our analysis, we have the sharper bound 
	\begin{align} \nonumber 
	\text{Var}_{f_A}\big(\hat f_A(a; h, \delta)|\sigma_Z\big) %
	& \leq \mbox{const.}\cdot h^{-3}\cdot \sumdel\,\big(Z_{(j+1)} - Z_{(j)} \big)^2\, 
	\end{align}
	since
	$ \int_\R K^2\big( u ;h\big)\, \di u \leq \mbox{const.}\cdot h^{-3}.$ Then one can show that our estimator also achieves the rate ${\cal O}\big(n^{-\frac{2 \alpha}{2\alpha+3}}\big)$ for $\beta=0$, even with the choice $\delta = 0$. 
	
Finally, we consider the uniform rate of convergence, again in the case $\beta >1$. 
\begin{thm} \label{T:3}
	Consider model (\ref{eq:model}) and assume that $f_X$ satisfies (\ref{eq:decayX}) for some $\beta>1$. Suppose that $w$ satisfies Assumption \ref{assum:kernel} for $l \geq 2\, \lfloor \alpha \rfloor$, and that $\delta = \delta_n$ and $h = h_n$ are chosen such that 
	\[ \delta \asymp \Big(\frac{\log n}{n}\Big)^{\frac1{\beta+1}}, \quad \text{ and }\quad h \asymp \Big(\frac{\log n}{n}\Big)^{\frac1{(\alpha+2)(\beta+1)}}.\]
	For a compact rectangle $K \subseteq \R^2$ let ${\cal F}(K,\alpha,c_A,c_B,r_A,c_M)$ denote the class of densities on $\R^2$ such that $f \in {\cal F}(a,\alpha,c_A,c_B,r_A,c_M)$ for each $a \in K$. Then 
	 the estimator (\ref{eq:est}) attains the following uniform asymptotic risk upper bound over the function class ${\cal F} = {\cal F}(K,\alpha,c_A,c_B,r_A,c_M)$, 
	$$ \sup_{f_A \in {\cal F}} \E_{f_A} \big[\sup_{a \in K}\,\big|\hat{f}_A(a;h,\delta) - f_A(a)\big|^2\big] \, = \, {\cal O}\Big(\Big(\frac{\log n}{n}\Big)^{\frac{2\, \alpha}{(\alpha+2)(\beta+1)}}\Big)\,. $$
\end{thm}

\section{Adaptation}\label{sec:adapt}

\subsection{Adaptation with respect to $\beta$ for given smoothness}

Assume that (\ref{eq:decayX}) holds with unknown $\beta >1$. 
If there are at least two observations $Z_j$ in the interval $[-\pi/2 + \delta, \pi/2 - \delta]$ so that $\sumdel$ is not the sum over the empty set, we set
\begin{align}\label{eq:lnwn}
L_{n}(\delta) & = \min\big\{ Z_j \mid Z_j \geq -\pi/2 + \delta\big\}, \quad
R_{n}(\delta)  = \max\big\{ Z_j \mid Z_j \leq \pi/2 - \delta\big\},
\end{align}
otherwise we put $L_n(\delta)=-\pi/2$ and $R_n(\delta) = \pi/2$.
To define a selection rule for $\delta$, define the function
\begin{align*}
{\cal C}_{n}(\delta) \, := \, & \, \sumdel\, (Z_{(j+1)}-Z_{(j)})^2 \, + \, \delta^{-1} \sumdel\, (Z_{(j+1)}-Z_{(j)})^3 \, \\ 
& \qquad \qquad \qquad  + \, (L_n(\delta) + \pi/2)^2 \, + \, (\pi/2 - R_n(\delta))^2 + \delta^2\,,
\end{align*}
which is continuous except at the sites $\pi/2$, $Z_j + \pi/2$ and $\pi/2 - Z_j$ for $j=1,\ldots,n$. Now choose $\delta = \hat{\delta}_n$ in the interval $[n^{-1/2},\pi/4]$ such that 
\begin{equation} \label{eq:hatdelta} 
{\cal C}_{n}(\hat{\delta}_n) \, \leq \, \exp(-n) + \inf_{\delta \in [n^{-1/2},\pi/4]} {\cal C}_{n}(\delta)\,. 
\end{equation}
%
%

The next proposition shows that the convergence rate from Theorem \ref{T:1} does not deteriorate if only $\beta$ is unknown but $\alpha$ is known.

\begin{prop} \label{P:1}
	Consider model (\ref{eq:model}) and assume that $f_X$ satisfies (\ref{eq:decayX}) for some unknown $\beta>1$. Choose $w$ satisfying the Assumption \ref{assum:kernel} for $2\, \lfloor \alpha \rfloor \leq l$ for given $\alpha>0$. If $\hat{\delta}_n$ is chosen in (\ref{eq:hatdelta}) and 
	\[ \hat h_n = \big({\cal C}_{n}(\hat{\delta}_n)\big)^{\frac1{2\,( \alpha +2)}},\]
	then for the estimator $\hat{f}_A\big(a;\hat h_n, \hat{\delta}_n\big)$ we have that 
	$$ \sup_{f_A \in {\cal F}} \E_{f_A} \big[\big|\hat{f}_A\big(a;\hat h_n, \hat{\delta}_n\big) - f_A(a)\big|^2\big] \, = \, {\cal O}\big(n^{-\frac{ 2\, \alpha}{(\alpha+2)(\beta+1)}}\big)\,, $$
	where ${\cal F} = {\cal F}(a,\alpha,c_A,c_B,r_A,c_M)$.
\end{prop}

\subsection{Adaptation by the Lepski method} \label{ss:Lepski}

Finally we consider adaptivity with respect to both parameters $\beta$ and $\alpha$ based on a combination of Lepski's method, see \cite{L91} and \cite{LS97}, and the choice (\ref{eq:hatdelta}). Consider the grid of bandwidths 
\[ h_k= \hat{\delta}^{1/2}_n\, q^k, \qquad k\in \cK_n = \{0, \ldots, K\},\]
where $q > 1$, $K = K_n = \lfloor \log_q n \rfloor$ and $\hat{\delta}_n$ is defined in (\ref{eq:hatdelta}). 
Fix $a \in \R^2$ and denote 
\[ \hat f_k = \hat f_A(a; h_k, \hat{\delta}_n).\]
For $\CLe >0$ sufficiently large to be chosen we let 
\begin{align*}
 \hat k & = \max\big\{k \in \cK_n \mid |\hat f_k - \hat f_l|^2 \leq \, \CLe \, \sigma(l,n) \quad \forall \ l \leq k  , \ l \in \cK_n\big\},
 \end{align*}
where
\[ \sigma(k,n) = h_k^{-4}\, C_n\big(\hat{\delta}_n\big)\, \log n, \quad k \in \cK_n.\] 
\begin{thm} \label{T:4}
	Consider model (\ref{eq:model}) and assume that $f_X$ satisfies (\ref{eq:decayX}) for some unknown $\beta>1$. Choose $w$ according to Assumption \ref{assum:kernel} for some $l \in \N_0$. Then for sufficiently large $\CLe>0$ (e.g.~$\CLe = 20^2$ suffices), we have, for every $ \alpha >0$ with $2\, \lfloor \alpha \rfloor \leq l$, that 
	$$ \sup_{f_A \in {\cal F}} \E_{f_A} \big[\big|\hat{f}_A\big(a;h_{\hat k}, \hat{\delta}_n\big) - f_A(a)\big|^2\big]  =  {\cal O}\Big(n^{-\frac{ 2\, \alpha}{(\alpha+2)(\beta+1)}} (\log n)^{\frac{\alpha}{\alpha + 2}}\, \Big)\,, $$
where ${\cal F} := {\cal F}(a,\alpha,c_A,c_B,r_A,c_M)$.
\end{thm}
Thus for adaptivity an additional logarithmic factor occurs in the pointwise rate under H\"older smoothness constraints.

\section{Proofs}\label{sec:proofs}

In the proofs we drop $f_A \in {\cal F}$ in $\E = \E_{f_A}$ and in $\Pb= \Pb_{f_A}$ from the notation. 

\subsection{Proofs for Section \ref{sec:bounds}}

\begin{proof}[Proof of Theorem \ref{T:1}]

By passing to Cartesian coordinates in (\ref{eq:regexpress}) we can write
\begin{align}\label{eq:regexpress1}
%
\tilde f_A(a; h) & = \frac{1}{(2 \pi)^2}\, \int_{\R^2} \exp\big( - i a'\,b \big) \psi_A (b) w(h\,\|b\|)    \di b  = \big(f_A * \tilde w(\cdot/h)/h^2\big)(a), \\ 
\tilde w (a) & = \frac{1}{(2 \pi)^2}\, \int_{\R^2}\, \exp\big( - i a'\,b \big)\,w(\|b\|)\,  \di b.\nonumber
\end{align}
Assumption \ref{assum:kernel} guarantees that $\tilde w$ is a kernel of order $\ell$.  Then, using Taylor approximation as usual in kernel regularization,
see p.~37--38 in \cite{M09} for the argument in case of non-compactly supported kernels, the following asymptotic rate of the regularization bias term occurs
\begin{align} \nonumber
\big|f_A(a) - \tilde f_A(a; h)\big| & \, = \, \Big|f_A(a) - \int \tilde w(z) f_A(a - h z) dz\Big| \\ \label{eq:bias}& \, \leq \, C_{\text{Bias}}(\alpha, w, c_A, c_M)\cdot h^\alpha\,,
\end{align}
where the constant factor $C_{\text{Bias}}(\alpha, w, c_A, c_M)$ only depends on $c_A$, $c_M$, $w$ and $\alpha$. 	

 Now let $\sigma_Z$ denote the $\sigma$-field generated by $Z_1, \ldots, Z_n$, and consider the conditional bias-variance decomposition 
\begin{align}\label{eq:biasvarconddecomp}
\begin{split}
\E \big[\big|\hat f_A(a; h, \delta) - \tilde f_A(a; h) \big|^2\big] = & \E \Big[\text{Var}\big(\hat f_A(a; h, \delta)|\sigma_Z \big) \Big] \\ & + \E \big[\big| \E\big[\hat f_A(a; h, \delta)| \sigma_Z\big] - \tilde f_A(a; h) \big|^2\big].
\end{split}
\end{align}
Since  $U_{[1]}, \ldots, U_{[n]}$ are independent  given $\sigma_Z$, observing from (\ref{eq:resulkernel}) that $\|K(\cdot;h)\|_\infty = {\cal O}(h^{-2})$,  we may bound 
\begin{align} \nonumber 
\text{Var}\big(\hat f_A(a; h, \delta)|\sigma_Z\big) %
& \leq \sumdel\,\big(Z_{(j+1)} - Z_{(j)} \big)^2 \\ \nonumber & \, \cdot \int_\R K^2\big( u - a_0 \cos Z_{(j)} - a_1 \sin Z_{(j)};h\big)\, f_{U|Z}\big(u|Z_{(j)}\big)\, \di u\notag\\\label{eq:variancebound} 
& \leq \mbox{const.}\cdot h^{-4}\cdot \sumdel\,\big(Z_{(j+1)} - Z_{(j)} \big)^2\, , 
\end{align}
where the constant factor only depends on $w$. Therein we use the notation (\ref{eq:sumshortnot}).   
For the conditional expectation, we obtain that
\begin{align*}
\E\big[\hat f_A(a; h, \delta)| \sigma_Z\big] & = \frac{1}{(2 \pi)^2}\, \int_\R\, w(th)\, |t|\, \, \int_{-\pi/2}^{\pi/2}\, \tilde \psi(t,z)\, \di z\, \di t
\end{align*}
where we set
\[ \tilde \psi(t,z) = \sumdel \psi_{U|Z}(t | Z_{(j)}) \, \exp\big(  - i t a_0 \cos Z_{(j)} - i t a_1 \sin Z_{(j)}\big) \, \ind(Z_{(j)} \leq z \leq Z_{(j+1)}) .\]

We deduce that
\begin{align} \label{eq:bias_estimate}
\big| \E\big[\hat f_A(a; h, \delta)| \sigma_Z\big] - \tilde f_A(a; h) \big|^2 & \, \leq \, I_1 + I_2 + I_3\,,
\end{align}
where 
\allowdisplaybreaks
\begin{align*}
I_1 \, := \, & \frac{3}{(2 \pi)^4}\, \Big|\int_{L_{n}(\delta)}^{R_{n}(\delta)}\, \int_\R\, w(th)\, |t|\,  \big( \tilde \psi(t,z) -  \exp\big( - i t (a_0 \cos z + a_1 \sin z) \big)\\ & \hspace{8.3cm} \cdot \psi_{U|Z} \big(t|z\big)\,\big)\, \di t\,  \di z\, \Big|^2\\
I_2 \, := \, &   \frac{3}{(2 \pi)^4}\, \Big|\int^{L_{n}(\delta)}_{- \pi/2}\, \int_\R\, w(th)\, |t|\,  \exp\big( - i t (a_0 \cos z + a_1 \sin z) \big)\\ & \hspace{8.6cm} \cdot \psi_{U|Z} \big(t|z\big)\,\di z\, \di t \Big|^2\\
I_3 \, := \,  &  \frac{3}{(2 \pi)^4}\, \Big|\int_{R_{n}(\delta)}^{ \pi/2}\, \int_\R\, w(th)\, |t|\,   \exp\big( - i t (a_0 \cos z + a_1 \sin z) \big)\\ & \hspace{8.4cm} \cdot \psi_{U|Z} \big(t|z\big)\,\di z\, \di t \Big|^2\,,
\end{align*}
where $L_{n}(\delta)$ and $R_{n}(\delta)$ are defined in (\ref{eq:lnwn}). 
If there are no two consecutive $Z_j$ in the interval $[-\pi/2 + \delta, \pi/2 - \delta]$, then $\tilde \psi(t,z)=0$ (indeed $\hat f_A(a; h, \delta)=0$). In this case, by our convention we have $L_n(\delta)=-\pi/2$ and $R_n(\delta) = \pi/2$ so that $I_2 = I_3 = $ and $I_1$ is the integral from $-\pi/2$ to $\pi/2$, as required for the estimate \eqref{eq:bias_estimate} to remain true in this case.

First, consider the term $I_3$. Using the Cauchy-Schwarz inequality, it holds that 
\allowdisplaybreaks
\begin{align}\label{eq:truncationterm}
\begin{split}
I_3 \, \leq \, &  \frac{3}{(2 \pi)^4}\, \int_{-1/h}^{1/h}\,t^2 \,  dt \, \int_{-1/h}^{1/h}\,  \Big|  \int_{R_{n}(\delta)}^{ \pi/2}\, \exp\big( - i t (a_0 \cos z + a_1 \sin z) \big)\\ & \hspace{8.4cm} \cdot \psi_{U|Z} \big(t|z\big)\,\di z\, \Big|^2\,  \di t \\
\leq \, &  \frac{4}{(2 \pi)^4}\cdot h^{-4} \cdot \big(\pi/2-R_{n}(\delta)\big)^2\,.
\end{split}
\end{align}
Analogously we establish that 
\begin{align*}
I_2 \leq  \frac{4}{(2 \pi)^4}\cdot h^{-4} \cdot \big(L_{n}(\delta) - \pi/2\big)^2\,.
\end{align*}
Finally, consider the term $I_1$. In case when there are two consecutive $Z_j$ in the interval $[-\pi/2 + \delta, \pi/2 - \delta]$ so that the sum in (\ref{eq:sumshortnot}) is not empty, it holds that
\begin{align} \nonumber
I_1 & \, \leq\, \frac{3}{(2\pi)^4} h^{-2} \cdot \Big\{ \sumdel \int_{Z_{(j)}}^{Z_{(j+1)}}\int\limits_{|t|\leq 1/h} \big|\tilde\psi(t,z) - \exp\big( - i t (a_0 \cos z + a_1 \sin z) \big)\\ \nonumber & \hspace{9cm} \cdot \psi_{U|Z} \big(t|z\big)\big| \di t\, \di z\Big\}^2
\end{align}
Now, for $z \in [Z_{(j)}, Z_{(j+1)} )$, we get that
\allowdisplaybreaks
\begin{align*}
\big|&\tilde\psi(t,z) - \exp\big( - i t (a_0 \cos z + a_1 \sin z) \big)\, \psi_{U|Z} \big(t|z\big)\big| \\
& = \big| \psi_{U|Z}(t | Z_{(j)}) \exp\big(  - i t a_0 \cos Z_{(j)} - i t a_1 \sin Z_{(j)}\big) -  \psi_{U|Z} \big(t|z\big)\\ & \hspace{7.2cm}\cdot \exp\big( - i t (a_0 \cos z + a_1 \sin z) \big) \big|\,\\ &
\leq \,  \big| \psi_{U|Z}(t | Z_{(j)}) - \psi_{U|Z}(t | z)\big| \, + \, |t| \cdot |a| \cdot (Z_{(j+1)} - Z_{(j)}) \\
& = \, \big| \psi_{A}(t \cos Z_{(j)}, t \sin Z_{(j)}) - \psi_{A}(t \cos z, t \sin z) \big| \, + \, |t| \cdot |a| \cdot \big(Z_{(j+1)} - Z_{(j)}\big)\,,
\end{align*}
according to (\ref{eq:relcharafct}). Hence we may bound
\allowdisplaybreaks
\begin{align*}
I_1 \,&\leq \text{const.}\,\cdot \,  h^{-2}  \, \Big(\Big\{\int_{-\pi/2}^{\pi/2}\, \int_{|t| \leq h^{-1}}\, \sumdel \ind(z \in [Z_{(j)}, Z_{(j+1)}]) \, \big| \psi_{A}(t \cos Z_{(j)}, t \sin Z_{(j)})\\
& \qquad  - \psi_{A}(t \cos z, t \sin z) \big|\, \di t\, \di z\, \Big\}^2\\
& \quad +  \, \Big\{\int_{-\pi/2}^{\pi/2}\, \int_{|t| \leq h^{-1}}\,  \sumdel \ind(z \in [Z_{(j)}, Z_{(j+1)}])\,  |t| \cdot |a| \cdot \big(Z_{(j+1)} - Z_{(j)}\big)\, \di t\, \di z\, \Big\}^2\Big)\\
& = \text{const.}\,\cdot \,  h^{-2} \big(I_{1,1} + I_{1,2}\big).
\end{align*}
Applying the Cauchy-Schwarz inequality gives for $I_{1,2}$
\[ I_{1,2} \leq \text{const.}\,\cdot \,  h^{-4}\, |a|^4 \, \sumdel\,\big(Z_{(j+1)} - Z_{(j)} \big)^3.\]
For $I_{1,1}$ interchanging sum and integrals we obtain 
\allowdisplaybreaks
\begin{align*}
& \int_{-\pi/2}^{\pi/2} \, \int_{|t|\leq 1/h}\,\ind(z \in [Z_{(j)}, Z_{(j+1)}])\, \big| \psi_{A}(t \cos Z_{(j)}, t \sin Z_{(j)}) - \psi_{A}(t \cos z, t \sin z) \big|\, \di t\, \di z\,\\
= & \int_{Z_{(j)}}^{Z_{(j+1)}}\int_{|t|\leq 1/h}  \big| \psi_{A}(t \cos Z_{(j)}, t \sin Z_{(j)}) - \psi_{A}(t \cos z, t \sin z) \big|\, \di t\, \di z \\
\leq & \int_{Z_{(j)}}^{Z_{(j+1)}}\int_{|t|\leq 1/h}  |t|\,\Big|\, \int_{Z_{(j)}}^z\,  \langle \nabla \psi_{A}(t \cos u, t \sin u), (-\sin u, \cos u)\rangle \, \di u \Big|\, \di t\, \di z\, \\
\leq & \int_{Z_{(j)}}^{Z_{(j+1)}}\int_{|t|\leq 1/h}  |t|\, \int_{Z_{(j)}}^z\,  \sup_{y \in \R}\, | \nabla \psi_{A}(t \cos u, y)| \,  \di u \, \di t\, \di z\,\\
\leq & 2\, h^{-1}\, \int_{Z_{(j)}}^{Z_{(j+1)}}  \, \int_{Z_{(j)}}^z\, \int_{t \in \R} \sup_{y \in \R}\, | \nabla \psi_{A}(t \cos u, y)| \, \di t\,   \di u \, \di z\\
\leq & 2\, c_B\, h^{-1}\, \int_{Z_{(j)}}^{Z_{(j+1)}}  \, \int_{Z_{(j)}}^z\, \frac{1}{\cos u }   \di u \, \di z\\ 
= & 2\, c_B\, h^{-1}\, \big(Z_{(j+1)}-Z_{(j)}\big)\, \int_{Z_{(j)}}^{Z_{(j+1)}}\, \frac{1}{\cos u }   \di u 
\end{align*}
Using the Cauchy-Schwarz inequality twice yields
\allowdisplaybreaks
\begin{align*}
I_{1,1} & \,\leq \text{const.}\,\cdot \,  h^{-2}\, \sumdel\, \big(Z_{(j+1)}-Z_{(j)}\big)\, \int_{Z_{(j)}}^{Z_{(j+1)}}\, \frac{1}{\cos u }  \di u\\
& \leq \text{const.}\,\cdot \,  h^{-2}\, \Big(\int_{-\pi/2+\delta}^{\pi/2-\delta}\, \sumdel \ind(z \in [Z_{(j)}, Z_{(j+1)}])\,  \,\big(Z_{(j+1)}-Z_{(j)}\big)\,    \frac{1}{\cos z } \,dz  \Big)^2\\
& \leq \text{const.}\,\cdot \, h^{-2}\, \int_{-\pi/2+\delta}^{\pi/2-\delta}\, \frac{1}{\cos^2 z }\, \di z\, \int_{-\pi/2+\delta}^{\pi/2-\delta}\, \sumdel \ind(z \in [Z_{(j)}, Z_{(j+1)}])\,  \,\big(Z_{(j+1)}-Z_{(j)}\big)^2\,    \,\di z \\
& \leq \text{const.}\,\cdot \,  h^{-2}\, \int_{-\pi/2+\delta}^{\pi/2-\delta}\, \frac{1}{\cos^2 z }\, \di z\, \sumdel \,\big(Z_{(j+1)}-Z_{(j)}\big)^3\,   \\ 
& \leq \text{const.}\,\cdot \,  \delta^{-1}\, h^{-2}\,  \sumdel \,\big(Z_{(j+1)}-Z_{(j)}\big)^3\,. 
\end{align*}
Hence, the term $I_1$ obeys the upper bound
\begin{align} \nonumber
I_1  \, \leq \, \mbox{const.}\cdot \big(|a|^2 \cdot h^{-6} + \delta^{-1}\, h^{-2} \big)\cdot \sumdel\,\big(Z_{(j+1)} - Z_{(j)} \big)^3. 
\end{align}
%
%
Finally, if there are no two consecutive $Z_j$ in the interval $[-\pi/2 + \delta, \pi/2 - \delta]$, we simply have $I_1 \leq \big| \tilde  f_A(a; h) \big|^2 \leq f_A(a)^2 + \text{const.}\,\cdot h^{2 \alpha} \leq \text{const.}$
Collecting the terms that bound \eqref{eq:bias_estimate} and using \eqref{eq:variancebound},  from \eqref{eq:biasvarconddecomp} we obtain that
\allowdisplaybreaks
\begin{align} 
\E &\big[\big|\hat f_A(a; h, \delta) - \tilde f_A(a; h) \big|^2\, \big|\, \sigma_Z\big] \notag\\
& \leq  \,\mbox{const.}\,\cdot h^{-4}\,   \Big\{  \big(\pi/2-R_{n}(\delta)\big)^2 + \big(L_{n}(\delta)+ \pi/2\big)^2 \notag\\
& \hspace{3.5cm}  \,+ \sumdel (Z_{(j+1)} - Z_{(j)})^2\, +\, \delta^{-1} \cdot  \sumdel (Z_{(j+1)} - Z_{(j)})^3 \Big\}\notag \\
& \,  \hspace{5cm} +\mbox{const.}\,\cdot \Big\{\, |a|^2 \,  h^{-6}\cdot  \sumdel \big(Z_{(j+1)} - Z_{(j)} \big)^3 \,\notag\\  
& \hspace{1.6cm} + \, \ind\big(Z_{(j)}< -\pi/2 + \delta \text{ or } Z_{(j+1)}> \pi/2 - \delta \quad \forall \ j=1, \ldots, n-1 \big)\Big\}.\label{est:variancecomp}
	\end{align}
Here, the last term takes care of the event in which the sum $\sumdel$ is empty and the estimator actually is zero. In order to bound the terms in \eqref{est:variancecomp} involving the order statistics, we note that since $\beta >1$, 
\[ \int_\delta^{\pi/2} u^{-\beta} \di u \, \asymp \delta^{1-\beta}, \qquad \int_\delta^{\pi/2}\, u^{-2\beta} \di u  \asymp \,\delta^{1-2\, \beta}.  \]
From (\ref{eq:bias}) and (\ref{est:variancecomp}) and Lemma \ref{lem:spacings} we obtain for $\delta \leq \pi/4$ that
\begin{align} \nonumber
& \E \big[\big|\hat f_A(a; h, \delta) - f_A(a) \big|^2\big] \\ \nonumber & \leq \mbox{const.}\cdot \Big\{ h^{2 \alpha} + h^{-4}\, \big(\delta + \frac{1}{c_Z\, n \, \delta^\beta}\big)^2 \,+ \, h^{-4}\, n^{-1}\, \delta^{1-\beta}\,+ h^{-4}\, \delta^{-1}\, n^{-2}\, \delta^{1-2 \beta}\,  \\  \nonumber
&   \hspace{3.5cm} +  h^{-6}\, n^{-2}\, \delta^{1-2\beta} + n \, \exp\big( - c_Z\, (n-1)\,(\pi/4)^\beta \big)\Big\}.
\end{align}
Upon inserting the rates for $\delta$ and $h$ we obtain the result. 

\end{proof}

\begin{proof}[Proof of Theorem \ref{T:low}]
 We introduce the functions
$$ f_{A,\theta}(a_0,a_1) \, := \, \alpha_n \beta_n f_0(\alpha_n a_0,\beta_n a_1) \, + \, c_L \cdot \theta \cdot \cos(2 \beta_n a_1) \cdot \alpha_n \beta_n \phi(\alpha_n a_0, \beta_n a_1)\,, $$
for $\theta \in \{0,1\}$, some constant $c_L>0$ and some sequences $(\alpha_n)_n \downarrow 0$ and $(\beta_n)_n \uparrow \infty$ which remain to be selected; moreover we specify  
$$ f_0(a_0,a_1) := \frac1{\pi^2 (1+a_0^2) (1+a_1^2)}\,, $$
and   
$$ \phi(a_0,a_1) := \varphi(a_0) \, \varphi(a_1)\,, $$
where  
$$ \varphi(x) \, := \, \frac{1 - \cos(x)}{\pi x^2}\,. $$
We verify that $f_{A,0}$ is a probability density as $f_0$ and $\varphi$ are probability densities. The Fourier transform of $f_{A,\theta}$ equals
\begin{align*} f_{A,\theta}^{ft}(x,y)  \, = \, f_0^{ft}(x/\alpha_n,y/\beta_n) & \, + \, \frac12 c_L \cdot \theta \cdot \phi^{ft}\big(x/\alpha_n,(y+2\beta_n)/\beta_n\big) \\ & \, + \, \frac12 c_L \cdot \theta \cdot \phi^{ft}\big(x/\alpha_n,(y-2\beta_n)/\beta_n\big)\,, \end{align*}
so that 
$$ \iint f_{A,\theta}(a_0,a_1) \di a_0 \, \di a_1 \, = \, f_{A,\theta}^{ft}(0) \, = \, f_{A,0}^{ft}(0) \, = \, 1\,, $$
since $\varphi^{ft}$ is supported on the interval $[-1,1]$. Choosing the constant $c_L>0$ sufficiently small we can guarantee that $f_{A,1}$ is a non-negative function and satisfies the inequality
\begin{equation} \label{eq:rev1} f_{A,1}(a_0,a_1) \, \geq \, c_L^* \alpha_n \beta_n f_0(\alpha_n a_0,\beta_n a_1) \, \geq \, 0\,, \qquad \forall a_0,a_1\in \mathbb{R}\,, \end{equation}
for some constant $c_L^* \in (0,1)$. Thus, $f_{A,1}$ is a probability density as well. Furthermore we verify that $f_{A,\theta} \in {\cal F}$ for both $\theta\in \{0,1\}$ under the constraint 
\begin{equation} \label{eq:alpha_n_beta_n} \alpha_n \asymp \beta_n^{-\alpha-1}\,, \end{equation}
as $c_A$ and $c_B$ may be viewed as sufficiently large. Therein note that (\ref{eq:Lipschitz}) is satisfied as $\psi_{A,\theta}$ can be written as the sum of two functions $(x,y) \mapsto \psi_0(x/\alpha_n) \cdot \psi_1(y/\beta_n)$ where $\psi_j$, $j=0,1$ are bounded, weakly differentiable, integrable functions whose weak derivatives are essentially bounded and integrable as well. 

The squared pointwise distance between $f_{A,0}$ and $f_{A,1}$ at $0$ equals
\begin{equation} \label{eq:lb.5} \big|f_{A,0}(0) - f_{A,1}(0)\big|^2 \, = \, c_L^2 \alpha_n^2 \beta_n^2 / (4\pi^2) \, \asymp \, \beta_n^{-2\alpha}\,. \end{equation}

Using (\ref{eq:rev1}), the conditional density of $Y_j$ given $X_j$ under the parameter $\theta$ equals 
	\allowdisplaybreaks
\begin{align*} f_{Y_j\mid X_j, \theta}(y) & \, = \, \int f_{A,\theta}(y-a_1 X_j,a_1) \di a_1 \, \geq \, c_L^* \alpha_n \beta_n \int f_0(\alpha_n (y-a_1 X_j),\beta_n a_1) \di a_1 \\
& \, \geq \, \frac{c_L^* \alpha_n \beta_n}{\pi^2} \int \frac1{1 + 2 \alpha_n^2 y^2 + 2 \alpha_n^2 a_1^2 X_j^2}\cdot \frac1{1 + \beta_n^2 a_1^2} \di a_1 \\
& \, \geq \, \frac{c_L^* \alpha_n \beta_n}{2\pi^2} \int_0^{1/\beta_n}  \frac1{1 + 2 \alpha_n^2 y^2 + 2 \alpha_n^2 a_1^2 X_j^2} \di a_1\\
& \, \geq \, \frac{c_L^*}{2\pi^2} \cdot \frac{\alpha_n}{1 + 2 \alpha_n^2 y^2 + 2 X_j^2 \alpha_n^2 / \beta_n^2}\,, \end{align*}
for all $y \in \mathbb{R}$. Moreover we have that
\begin{align*}  f_{Y_j\mid X_j, 1}(y) -  f_{Y_j\mid X_j, 0}&(y)  \, = \, c_L \alpha_n \beta_n \int \cos(2\beta_n a_1) \cdot \phi\big(\alpha_n (y - a_1 X_j), \beta_n a_1\big) \di a_1 \\
& \, = \, c_L \alpha_n \int \cos(2 a_1 \beta_n / \beta_n) \cdot \phi\big(\alpha_n (y - a_1 X_j /\beta_n), a_1\big) \di a_1\,,
\end{align*}
where the Fourier transform equals
\begin{align*} f_{Y_j\mid X_j, 1}^{ft}(t) -  f_{Y_j\mid X_j, 0}^{ft}(t) & \, = \, \frac12 c_L \phi^{ft}\big(t/\alpha_n, (tX_j + 2\beta_n)/\beta_n\big) \\ &\, + \, \frac12 c_L \phi^{ft}\big(t/\alpha_n, (tX_j - 2\beta_n)/\beta_n\big)\,. \end{align*}
Therefore the $\chi^2$-distance between the competing observation densities is bounded from above as follows, 
\begin{align} \nonumber
c_L^*\cdot \chi^2\big(&f_{Y_j\mid X_j, \theta=0},f_{Y_j\mid X_j, \theta=1}\big) \\ \nonumber & \, \leq \, (1/\alpha_n + 2X_j^2 \alpha_n/\beta_n^2) c_L^2 \int \big|\phi^{ft}\big(t/\alpha_n, (tX_j + 2\beta_n)/\beta_n\big)\big|^2 \di t \\ \nonumber & + (1/\alpha_n + 2X_j^2 \alpha_n/\beta_n^2) c_L^2 \int \big|\phi^{ft}\big(t/\alpha_n, (tX_j - 2\beta_n)/\beta_n\big)\big|^2 \di t \\ \nonumber 
& \, + \, 2 c_L^2 \alpha_n \, \int \Big|\frac{d}{dt}\phi^{ft}\big(t/\alpha_n, (tX_j + 2\beta_n)/\beta_n\big)\Big|^2 \di t \\ \label{eq:lb.10} & + 2 c_L^2 \alpha_n \, \int \Big|\frac{d}{dt}\phi^{ft}\big(t/\alpha_n, (tX_j - 2\beta_n)/\beta_n\big)\Big|^2 \di t\,,
\end{align}
where 
\allowdisplaybreaks
\begin{align*} \frac{d}{dt} \phi^{ft}\big(t/\alpha_n, (tX_j \pm 2\beta_n)/\beta_n\big) & \, = \,  \alpha_n^{-1}\, \big\{\varphi^{ft}\big\}'\big(t/\alpha_n\big) \cdot  \big\{\varphi^{ft}\big\}\big((tX_j \pm 2\beta_n)/\beta_n\big) \\ & + \frac{X_j}{\beta_n}\cdot  \big\{\varphi^{ft}\big\}\big(t/\alpha_n\big) \cdot  \big\{\varphi^{ft}\big\}'\big((tX_j \pm 2\beta_n)/\beta_n\big)\,. \end{align*}
Moreover, this choice also guarantees that $f_{C,\theta}$ integrates to $1$ and, hence, is a probability density. Then the integrals in (\ref{eq:lb.10}) range over a subset of $$[-\alpha_n,\alpha_n] \backslash (-\beta_n/|X_j|,\beta_n/|X_j|)$$ as $H_0^{ft}$ and its (weak) derivative are supported on $[-1,1]$. Also these functions are uniformly bounded by $1$. Thus the integrals vanish whenever $|X_j| < \beta_n / \alpha_n$. It follows that
\begin{align} \nonumber
\chi^2\big(f_{Y_j\mid X_j, \theta=0},f_{Y_j\mid X_j, \theta=1}\big) & \, \leq \, (6 + 8 X_j^2 \alpha_n^2/\beta_n^2) c_L^2 / c_L^*\,,
\end{align}
if $|X_j| \geq \beta_n / \alpha_n$; and $\chi^2(f_{Y_j\mid X_j, \theta=0},f_{Y_j\mid X_j, \theta=1}) = 0$ otherwise. According to standard arguments from decision theory, (\ref{eq:lb.5}) represents a lower bound on the attainable rate if the Hellinger distance between the competing data distributions $f_{X,Y;\theta}^{(n)}$ (for $\theta=0$ and $\theta_1$, respectively) obeys an upper bound which is smaller than $1$ -- uniformly with respect to $n$, see e.g. \cite{T09}. Writing ${\cal H}$ for the Hellinger distance, it holds that 
\begin{align} \nonumber
{\cal H}^2\big(f_{X,Y;\theta=0}^{(n)},f_{X,Y;\theta=1}^{(n)}\big) & \, \leq \, \sum_{j=1}^n {\cal H}^2\big(f_{X_j}f_{Y_j|X_j,0},f_{X_j}f_{Y_j|X_j,1}\big) \\\nonumber
& \, = \,  \sum_{j=1}^n \int f_{X_j}(x) \int \big(f_{Y_j|X_j,0}^{1/2}(y|x) - f_{Y_j|X_j,1}^{1/2}(y|x)\big)^2 \di y \, \di x \\ \label{eq:rev2}
& \, \leq \, \E \, \Big[\sum_{j=1}^n  \chi^2(f_{Y_j\mid X_j, 0},f_{Y_j\mid X_j, 1})\Big] \end{align}
as the distribution of the $X_j$ is identical for $\theta=0$ and $\theta=1$. Then, the term (\ref{eq:rev2}) is bounded from above by
\begin{align} \nonumber 
 \E & \Big[\sum_{j=1}^n 1_{[\beta_n/\alpha_n,\infty)}(|X_j|) \cdot (6 + 8 X_j^2 \alpha_n^2/\beta_n^2) c_L^2\Big] \\ \nonumber 
& \, = \, 6 \, n \, c_L^2 \, \int_{|x|\geq \beta_n/\alpha_n} f_X(x) \di x \, + \, 8 n c_L^2\, \alpha_n^2\, \beta_n^{-2}\, \int_{|x|\geq \beta_n/\alpha_n} x^2 f_X(x) \di x \\ \label{eq:chi2}
& \, = \, {\cal O}\big(n (\beta_n/\alpha_n)^{-\beta-1}\big) \, = \, {\cal O}\big(n \cdot \beta_n^{-(\alpha+2)(\beta+1)}\big)\,, \end{align}
as $\beta>1$. 
We choose $\beta_n \asymp n^{1/[(2+\alpha)(1+\beta)]}$ so that the $\chi^2$-distance between the joint densities of the observations under $\theta=0$ and $\theta=1$ in (\ref{eq:chi2}) is bounded from above as $n$ tends to infinity. By elementary decision theoretic arguments and by (\ref{eq:lb.5}), a lower bound on the attainable convergence rate is given by
$$ \beta_n^{-2\alpha} \asymp n^{-\frac{2\alpha}{(\alpha+2)(\beta+1)}}\,, $$
which completes the proof of the theorem. 
\end{proof}

\begin{proof}[Proof of Theorem \ref{T:3}]
We estimate
\begin{align*}
\E \big[\sup_{a \in K} \big|\hat f_A(a; h, \delta) - f_A(a)  \big|^2 \big]\, \leq &\, 2\, \E \big[ \sup_{a \in K} \big|\hat f_A(a; h, \delta) - \tilde  f_A(a;h)  \big|^2\big]\\
& \,  + 2\, \sup_{a \in K} \big|\tilde f_A(a;h)  - f_A(a)\big|^2,
\end{align*}
where $\tilde f_A(a;h)$ is defined in \eqref{eq:regexpress1}. 
The second term - the regularization bias - is bounded in (\ref{eq:bias}), and that bound is uniform in $a \in K$ from the assumptions on the function class ${\cal F}(K,\alpha,c_A,c_B,r_A,c_M)$.  
For the first term we have, similarly to \eqref{eq:biasvarconddecomp}, that
\begin{align}\label{eq:decompuniform}
\begin{split}
\E \big[ \sup_{a \in K} \big|\hat f_A(a; h, \delta) - \tilde f_A(a; h, \delta) \big|^2\big] & \leq 2\,\E \big[ \sup_{a \in K} \big|\hat f_A(a; h, \delta) - \E \big[\hat f_A(a; h, \delta)| \sigma_Z\big] \big|^2 \big]\\ 
& + 2\, \E \big[\sup_{a \in K} \big|  \E \big[\hat f_A(a; h, \delta)| \sigma_Z\big] - \tilde f_A(a; h) \big|^2\big].
\end{split}
\end{align}
The second term in (\ref{eq:decompuniform}) is bounded by
\begin{align*}
& \sup_{a \in K} I_1(a) \, + \, \sup_{a \in K} I_2(a) \, + \, \sup_{a \in K} I_3(a),   
\end{align*}
where $I_j(a)$ are defined as in (\ref{eq:bias_estimate}), and the dependence on $a$ is stressed in the notation. 
The bounds on the $I_j(a)$ derived after (\ref{eq:bias_estimate}) are uniform in $a$ over a bounded set $K$. Thus, it remains to bound the first term in (\ref{eq:decompuniform}).

Given $\epsilon>0$ let $I_\epsilon$ be a subset of $K$ for which the $\epsilon$-balls with centers at points in $I_\epsilon$ cover $K$. It is possible to choose such a set with a cardinality of order $\text{card}\,( I_\epsilon) \leq C_K\, \epsilon^{-2}$, where $c_K>0$ depends on $K$ but not on $\epsilon$. Then
\begin{align*}
&  \sup_{a \in K} \big|\hat f_A(a; h, \delta) - \E \big[\hat f_A(a; h, \delta)| \sigma_Z\big] \big|^2
\leq  2\, \sup_{a \in I_\epsilon} \big|\hat f_A(a; h, \delta) - \E \big[\hat f_A(a; h, \delta)| \sigma_Z\big] \big|^2\\
&  + \, 2\,\sup_{a \in K} \inf_{a' \in I_\epsilon}\, \big|\hat f_A(a; h, \delta) - E\big[\hat f_A(a; h, \delta)| \sigma_Z\big] - \big(\hat f_A(a'; h, \delta) - E\big[\hat f_A(a'; h, \delta)| \sigma_Z\big]\big) \big|^2.
\end{align*}
Since $ \|\partial_x\, K(x;h)\|_\infty  \leq \,  h^{-3}$, see the formula \eqref{eq:resulkernel} for $K(\cdot;h)$ and the Assumption \ref{assum:kernel} in $w$, by Lipschitz-continuity the second term is $\leq 8\, \epsilon^2 \, h^{-6}$.     
From the Hoeffding inequality, since $\|K(\cdot;h)\|_\infty \leq h^{-2}$ we obtain for $t>0$ that
\begin{align}\label{eq:variancebound1}
\Pb \Big(\big|\hat f_A(a; h, \delta) - \E \big[\hat f_A(a; h, \delta)| \sigma_Z\big] \big| \geq t \Big|\sigma_Z\Big)
\leq \, 2\, \exp\Big(-\frac{t^2}{2 \,  h^{-4}\,  \sumdel \big(Z_{(j+1)} - Z_{(j)} \big)^2\,} \Big).
\end{align}
Set 
\[ r_n = (\log n) \, \cdot \, h^{-4}\, \sumdel \big(Z_{(j+1)} - Z_{(j)} \big)^2.\]
Then, for $\kappa>0$ we estimate
\begin{align*}
&\E\Big[ r_n^{-1}\, \sup_{a \in I_\epsilon} \big|\hat f_A(a; h, \delta) - \E \big[\hat f_A(a; h, \delta)| \sigma_Z\big] \big|^2\, \Big| \sigma_Z\Big] \\
\leq &  \kappa^2 + 2 \, \int_\kappa^\infty\,  t\,\Pb \Big(\sup_{a \in I_\epsilon}\big|\hat f_A(a; h, \delta) - \E \big[\hat f_A(a; h, \delta)| \sigma_Z\big] \big| \geq r_n^{1/2}\, t \Big|\sigma_Z\Big) \di t \\
\leq &  \kappa^2 + 4\, \text{card}\,( I_\epsilon)\, \int_\kappa^\infty\,  t\, \exp\Big(-\frac{t^2\, \log n}{2} \Big)\, \di t \tag{from \eqref{eq:variancebound1} and union bound }\\
\leq & \kappa^2 + 4\, C_K\, \epsilon^{-2}\, \exp\Big(-\frac{\kappa^2\, \log n}{2} \Big)\, (\log n)^{-1}.
\end{align*}
Choose $\epsilon = n^{-2}$ and $\kappa = 10^{1/2}$. Then if $h^{-1} = {\cal O} (n^{1/2})$ we obtain from Lemma \ref{lem:spacings} that
\[ \E\Big[ \sup_{a \in I_\epsilon} \big|\hat f_A(a; h, \delta) - \E \big[\hat f_A(a; h, \delta)| \sigma_Z\big] \big|^2\, \Big] = {\cal O} \big(\E[r_n] \big) = h^{-4}\, \delta^{1-\beta}\, \nicefrac{\log n}{n} , \]
and overall
\begin{align} \nonumber
& \E \big[\sup_{a \in K} \big|\hat f_A(a; h, \delta) - f_A(a)  \big|^2 \big]\, \\ \nonumber & \leq \mbox{const.}\cdot \Big\{ h^{2 \alpha} +  \, h^{-4}\, \delta^{1-\beta}\,  \nicefrac{\log n}{n} + n^{-1}  + \, h^{-4}\, \big(\delta \, + \frac{1}{c_Z\, n \, \delta^\beta}\big)^2 \,+ h^{-4}\, \delta^{-1}\, n^{-2}\, \delta^{1-2 \beta}\,  \\  \nonumber
&   \hspace{3.5cm} +  h^{-6}\, n^{-2}\, \delta^{1-2\beta} + n \, \exp\big( - c_Z\, (n-1)\,(\pi/4)^\beta \big)\Big\}.
\end{align}
Plugging in the choices of $\delta$ and $h$ gives the result. 
\end{proof}


\subsection{Proofs for Section \ref{sec:adapt}}

\begin{proof}[{\sl Proof of Proposition \ref{P:1}}]
	From (\ref{est:variancecomp}) and (\ref{eq:bias}) we estimate
	\begin{align} 
 	& \E \big[\big|\hat f_A\big(a;\hat h_n, \hat{\delta}_n\big) - f_A(a) \big|^2\, \big|\, \sigma_Z\big] \notag \\
	& \leq   \mbox{const.}\,\cdot\,\big\{\hat h_n^{2\alpha} \, + \, \hat h_n^{-4}\, {\cal C}_{n}(\hat{\delta}_n)\big\} 
	+ \,  \mbox{const.}\,\cdot\, \Big\{\, |a|^2 \,  \hat{h}_n^{-6}\, \cdot  \sum_{j,n,\hat{\delta}_n} \big(Z_{(j+1)} - Z_{(j)} \big)^3 \,\notag\\  
	& \hspace{1.5cm} + \, \ind\big(Z_{(j)}< -\pi/2 + \hat{\delta}_n \text{ or } Z_{(j+1)}> \pi/2 - \hat{\delta}_n \quad \forall \ j=1, \ldots, n-1 \big)\Big\}. \label{eq:choosedeltbound}
	\end{align}
	Observe that from the term $\delta^2$ in the definition of $ {\cal C}_{n}(\delta)$, 
	\[ \hat h_n^2 = \big({\cal C}_{n}(\hat{\delta}_n)\big)^{\frac1{ \alpha +2}} \geq \hat{\delta}_n^{\frac{2}{ \alpha +2}} \geq \hat{\delta}_n.\]	
	Since $\hat{\delta}_n \leq \pi/4 \leq 1$, and since ${\cal C}_n(\delta)$ contains the term $\delta^{-1} \sumdel\, (Z_{(j+1)}-Z_{(j)})^3$, from \eqref{eq:choosedeltbound} and the choice of $\hat h_n$ we obtain the bound
	\begin{align} 
	& \E \big[\big|\hat f_A\big(a;\hat h_n, \hat{\delta}_n\big) - f_A(a) \big|^2\, \big|\, \sigma_Z\big]\notag \\
	& \leq    \mbox{const.}\cdot \Big\{ \big[{\cal C}_{n}(\hat{\delta}_n) \big]^{\frac{\alpha}{\alpha+2}}
	+  \ind\big(Z_{(j)}< -\pi/4 \text{ or } Z_{(j+1)}> \pi/4 \quad \forall \  j=1, \ldots, n-1 \big)\Big\}.  \label{eq:helpdelt}
	\end{align}
	By definition of $\hat \delta_n$, 
	\[ {\cal C}_{n}(\hat{\delta}_n) \leq \exp(-n) + \inf_{\delta \in [n^{-1/2},\pi/4]} {\cal C}_{n}(\delta) \leq \exp(-n) + {\cal C}_{n}(\delta_n) \]
	for the deterministic choice $\delta_n = n^{-1/(\beta +1)}$, which is contained in $[n^{-1/2},\pi/4]$ for sufficiently large $n$ since $\beta >1$. 
	Further, by Jensen's inequality, Lemma \ref{lem:spacings} and the choice of $\delta_n$, 
	\[ \E \Big[\,\big({\cal C}_{n}(\delta_n) \big)^{\frac{\alpha}{\alpha+2}}\Big] \leq  \big(\E \big[{\cal C}_{n}(\delta_n) \big]\big)^{\frac{\alpha}{\alpha+2}} = {\cal O}\big(n^{-\frac{ 2\, \alpha}{(\alpha+2)(\beta+1)}}\big).\]
	Substituting these estimates into (\ref{eq:helpdelt}), and using (\ref{eq:probcontains}) finally gives
	\begin{align*}
	 \E \big[\big|\hat f_A\big(a;\hat h_n, \hat{\delta}_n\big) - f_A(a) \big|^2\big]\,  \leq\,&  \,{\cal O}\big(n^{-\frac{ 2\, \alpha}{(\alpha+2)(\beta+1)}}\big) \,+\, \mbox{const.}\, \Big\{ \, \big[\exp(-n)\big]^{\frac{\alpha}{\alpha+2}}  \\
	\,+\,  \Pb\big(Z_{(j)}< -\pi/4 \text{ or } & Z_{(j+1)}> \pi/4 \quad \forall j=1, \ldots, n-1 \big)\Big\}\\
	 = \, &\,   {\cal O}\big(n^{-\frac{2\, \alpha}{(\alpha+2)(\beta+1)}}\big).
	\end{align*}
\end{proof}

\begin{proof}[Proof of Theorem \ref{T:4}]
	Fix $0 < \alpha$ with $2\, \lfloor \alpha \rfloor \leq l$  and $f_A \in {\cal F}(a,c_A,c_B,r_A,\alpha, c_M)$, and set 
	\[ b(k, \alpha) = C_{\text{Bias}}^2(\alpha, w, c_A,c_M)\, h_k^{2 \alpha}, \qquad k \in \cK_n,\]
	see the  bound for the regularization bias in (\ref{eq:bias}). 
	We shall abbreviate $f_A(a) = f$. 

	On the event 
	\[ \{ Z_{(j)}< -\pi/2 + \hat \delta_n \text{ or } Z_{(j+1)}> \pi/2- \hat \delta_n  \quad \forall \ j=1, \ldots, n-1 \}\]
	where $\hat{f}_{\hat k} = 0$, we may estimate
	\[ \E \big[\big|\hat{f}_{\hat k} - f\big|^2 \, \big|\, \sigma_Z\big] \leq \text{const.}\, \cdot \,\ind\big(Z_{(j)}< -\pi/4 \text{ or } Z_{(j+1)}> \pi/4  \quad \forall \ j=1, \ldots, n-1 \big) \]
	since $\hat \delta_n \leq \pi/4$. In the following, suppose that there are two design points $Z_j$ in the interval 
	$[-\pi/2 + \hat \delta_n, \pi/2 - \hat \delta_n]$. Since $h_k \geq \hat{\delta}^{1/2}_n$ for each $k \in \cK_n$, as in the proof of Proposition \ref{P:1} the term involving $h_k^{-6}$ in \eqref{eq:choosedeltbound} is negligible as compared to that with the factor $\hat \delta_n^{-1}\,h_k^{-4}$. Hence using (\ref{est:variancecomp}) and (\ref{eq:bias})  we estimate
	\begin{align} 
	&\, \E \big[\big|\hat f_k - f \big|^2\, \big|\, \sigma_Z\big] 
	\leq \  \mbox{const.}\,\cdot\,\big\{ b(k,\alpha) \, + \, \sigma(k,n)\big\}. \label{eq:firstlep}
	\end{align}
	Define the `oracle index' $k^*$ by
	\[ k^* = k_n^*(\alpha) = \max\big\{k \in \cK_n \mid b(k,\alpha) \leq \CLe \sigma(k,n)/16\big\}.\]
	Note that $b(0,\alpha) = C_{\text{Bias}}^2(\alpha, w, c_A, c_M) \, \hat \delta_n^\alpha \leq \text{const.}$ since $\hat \delta_n^\alpha \leq 1$, while 
	$ \sigma(0,n) = \delta_n^{-2}\,C_n(\hat \delta_n) \,\log n \geq \,\log n$ since ${\cal C}_n(\hat \delta_n) \, \hat \delta_n^{-2} \geq 1$ from the  definition of ${\cal C}_n(\delta)$.
	Further, since by the choice of $K$ we have that $q^K \geq n/q$ we estimate 
	\[ b(K,\alpha) \geq {\cal C}_{\text{Bias}}^2(\alpha, w, c_A, c_M) \, \hat \delta_n^\alpha\,\big(\nicefrac{n}{q}\big)^{2 \alpha} \geq \text{const.}\, n^{3 \alpha/2}\]
	 since $ \hat \delta_n^\alpha \geq n^{-\alpha/2}$ by the choice of $\hat \delta_n$. Finally,   
	\[ \sigma(K,n) \leq \delta_n^{-2}\,\,\big(\nicefrac{q}{n}\big)^{4}\, {\cal C}_n(\hat \delta_n) \,(\log n) \leq \text{const.}\, n^{- 3/2}\, \log(n),\]	
	since ${\cal C}_n(\hat \delta_n) \, \hat \delta_n^{-2} \leq \text{const.}\, \cdot n^{5/2}$ since from the definition of ${\cal C}_n(\delta)$ and since $\hat \delta_n \geq n^{-1/2}$.
	
	Since $b(k, \alpha)$ increase by factors $q^{2 \alpha}$ in $k$, and $\sigma(k,n)$ decrease by factors $q^{-4}$ in $k$, it follows from the above estimates that $k^* \to \infty$ and $K-k^* \to \infty$, and that there are constants $0 < \tilde c_1 < \tilde c_2$ such that $\tilde c_1 \leq \sigma(k^*,n)/b(k^*, \alpha) \leq \tilde c_2$. Rearranging yields 
	%
	%
%
%

	%
	\begin{equation}\label{eq:oralceindexprop}
	c_1 \, \big({\cal C}_{n}(\hat{\delta}_n)\, \log n\big)^{\frac1{2\,( \alpha +2)}}\leq h_{k^*} \leq c_2\, \big({\cal C}_{n}(\hat{\delta}_n)\, \log n\big)^{\frac1{2\,( \alpha +2)}}
	\end{equation}
	for constants $c_2 > c_1 > 0$. 
	We obtain from (\ref{eq:firstlep}) that
	{\small 
		\begin{align}\label{eq:oraclest}
		\E \big[\big|\hat f_{k^*} - f\big|^2\, \big|\, \sigma_Z\big] \leq \, \mbox{const.}\,\cdot\,  \big[{\cal C}_{n}(\hat{\delta}_n) \, \log n\big]^{\frac{\alpha}{\alpha+2}}.
		\end{align}
	} 
	Now, for $\hat f_{\hat k}$ we estimate
	\begin{align}\label{eq:Lep1}
	\begin{split}
	\E \big[\big|\hat{f}_{\hat k} - f\big|^2 \, \big|\, \sigma_Z\big]\leq & \, 2\, \E \big[\big|\hat{f}_{\hat k} - f\big|^2 \,\ind(\hat k \leq k^*-1) \big|\, \sigma_Z\big]  \\
	+\, & 2\, \E \big[\big|\hat{f}_{\hat k} - f\big|^2 \,\ind(\hat k \geq k^*) \big|\, \sigma_Z\big].
	\end{split}
	\end{align}
	For the second term, we have that 
	\begin{align}\label{eq:boundoverest}
	\begin{split}
	\E \big[\big|\hat{f}_{\hat k} - f\big|^2 \,1(\hat k \geq k^*) \big|\, \sigma_Z\big]& \leq 2\,\E \big[\big|\hat{f}_{\hat k} - \hat f_{k^*}\big|^2 \,\ind(\hat k \geq k^*) \big|\, \sigma_Z\big]\\
	& + 2 \, \E \big[\big|\hat f_{k^*} - f\big|^2 \,\ind(\hat k \geq k^*) \big|\, \sigma_Z\big].
	\end{split}
	\end{align}
	The second term in (\ref{eq:boundoverest}) is bounded by (\ref{eq:oraclest}) after a trivial estimate of the indicator. 
	Further, from the definition of $\hat k$ and (\ref{eq:oralceindexprop}) we have the bound
	\[ \big|\hat{f}_{\hat k} - \hat f_{k^*}\big|^2\, 1_{\hat k \geq k^*} \leq \CLe \, \sigma(k^*,n)  \leq \, \text{const.}\, \cdot  \big[{\cal C}_{n}(\hat{\delta}_n) \, \log n\big]^{\frac{\alpha}{\alpha+2}}, \]
	which also holds in conditional expectation given $\sigma_Z$. 
	
	For the first term in (\ref{eq:Lep1}) we estimate
	\begin{align}
	\E \Big[\big|\hat{f}_{\hat k} - f\big|^2\, \ind(\hat k \leq k^* - 1)\,\big|\, & \sigma_Z\Big]	=  \sum\limits_{k=0}^{k^*-1}  \E \Big[\big|\hat{f}_k - f\big|^2\, \ind(\hat k =k)\,\big|\, \sigma_Z\Big] \notag\\
	&\leq \sum_{k=0}^{k^*-1} \Big( \E \Big[ \big|\hat{f}_k - f\big|^4\, \,\big|\, \sigma_Z\Big] \Big)^{1/2} \,  \Big[\Pb\big(  \hat{k}=k \,\big|\, \sigma_Z \big) \Big]^{1/2}.\label{eq:boundearlystop}
	\end{align}
	Then
	$$  \{\hat{k}=k\} \subseteq \bigcup_{l=0}^k \Big\{ \big| \hat{f}_{k+1} -\hat{f}_{l}   \big|^2 > \CLe \, \sigma(l,n)  \Big\}, \qquad k=0, \ldots, K-1.		$$
	Now let  
	\[
	p_{l,k}=\Pb\big( | \hat{f}_{k} - \hat{f}_l     | > \CLe^{1/2}\, (\sigma(l,n))^{1/2} \,\big|\, \sigma_Z \big),\qquad 0\leq l < k \leq k^*.
	\]
	By choice of $k^*$, for $0\leq l < k \leq k^*$ we have that
	\[ b(l,\alpha) \leq b(k, \alpha) \leq \CLe\,\sigma(k,n)/16 \leq \CLe\,\sigma(l,n)/16. \]
	Hence, setting 
	$ \tilde f_k = \tilde f_A(a; h_k)$ we may estimate
	\begin{align*}
	| \hat{f}_{k} - \hat{f}_l | \, \leq & \, | \hat{f}_{k} - \tilde{f}_{k} | + | \hat{f}_{l} - \tilde{f}_{l} |
	+ | \tilde{f}_{k} - f| + | \tilde{f}_{l} -f|\\
	\leq \, & \, | \hat{f}_{k} - \tilde{f}_{k} | + | \hat{f}_{l} - \tilde{f}_{l} | + b(k, \alpha)^{1/2} + b(l, \alpha)^{1/2}\\
	\leq \, & \, | \hat{f}_{k} - \tilde{f}_{k} | + | \hat{f}_{l} - \tilde{f}_{l} | + \CLe^{1/2} \sigma(l,n)^{1/2}/2.
	\end{align*}
	Therefore, for $0\leq l < k \leq k^*$, 
	\begin{align*}
	p_{l,k} \leq & \Pb\big( | \hat{f}_{k} - \tilde{f}_{k} |  > \CLe^{1/2}  \sigma(l,n)^{1/2}/4 \,\big|\, \sigma_Z \big) \\ & + \Pb_{f_A}\big( | \hat{f}_{l} - \tilde{f}_{l} | > \CLe^{1/2} \sigma(l,n)^{1/2}/4 \,\big|\, \sigma_Z \big). \end{align*}
	Since $\sigma(l,n) > \sigma(k,n)$, $l < k$, it suffices to bound 
	\[ \Pb\big( | \hat{f}_{l} - \tilde{f}_{l} | > \CLe^{1/2} \sigma(l,n)^{1/2}/4 \,\big|\, \sigma_Z \big), \qquad 0 \leq l \leq k^*.\] 
	By choice of the grid $\cK_n$, $h_l^2 \geq h_0^2 = \hat \delta_n$, therefore 
	\[ \big| \E\big[ \hat{f}_{l}   \,\big|\, \sigma_Z \big] - \tilde{f}_{l} \big| \leq \mbox{const.}\,\cdot\, \big[h_l^{-4}\, C_n\big(\hat{\delta}_n\big) \big]^{1/2} \leq \, \sigma(l,n)^{1/2}  \]
	%
	%
	for $n$ sufficiently large. Hence
	\[ \Pb\big( | \hat{f}_{l} - \tilde{f}_{l} | > \CLe^{1/2} \sigma(l,n)^{1/2}/4 \,\big|\, \sigma_Z \big) \leq \Pb\big( | \hat{f}_{l} - \E\big[ \hat{f}_{l}   \,\big|\, \sigma_Z \big] | > \tilde C\, \sigma(l,n)^{1/2} \,\big|\, \sigma_Z \big),\]
	where $\tilde C = \big(\CLe^{1/2}/4 \, - 1\big)$. Using the bound $\|K(\cdot;h)\|_\infty \leq \, h^{-2}$,  
	 see the formula \eqref{eq:resulkernel} for $K(\cdot;h)$ and the Assumption \ref{assum:kernel} in $w$, we use the conditional Hoeffding inequality in order to estimate 
	\begin{align*}
	\Pb\big( | \hat{f}_{l} - \E\big[ \hat{f}_{l}   \,\big|\, \sigma_Z \big] | >  \tilde C\, & \sigma(l,n)^{1/2} \,\big|\, \sigma_Z \big) \\ & \leq 2\, \exp\Big(-\frac{\tilde C^2\, \sigma(l,n)}{2\, h_l^{-4}\, \sum_{j,n,\hat{\delta}_n} \big(Z_{(j+1)} - Z_{(j)} \big)^2} \Big)\\
	& \leq 2 \exp(- \bar C\, \log n),
	\end{align*}
	see (\ref{eq:variancebound}), where
	\[ \bar C = \tilde C^2/2 = \big(\CLe^{1/2}/4 \, - 1\big)^2/2 = 8\]
	for the choice $\CLe = 20^2$. Note that in this step, the logarithmic factor is essential.

	Hence
	\[ \Pb\big(\hat{k}=k\,\big|\, \sigma_Z ) \leq 2\, K\,n^{- 8}, \qquad k=0, \ldots, k^*,\]
	and in (\ref{eq:boundearlystop}) we obtain the bound
	\begin{align}
	\E \Big[\big|\hat{f}_{\hat k} - f\big|^2\, \ind(\hat k \leq k^* - 1)\,\big|\, \sigma_Z\Big]	&\leq\, 2\, K^{1/2}\, n^{-8/2}\,  \sum_{k=0}^{k^*-1} \Big( \E \big[ \big|\hat{f}_k - f\big|^4\, \,\big|\, \sigma_Z\big] \Big)^{1/2} \,  .
	\end{align}
	The crude bound
	\begin{align*}
	\E \big[ \big|\hat{f}_k - f\big|^4\, \,\big|\, \sigma_Z\big] & \leq \E \big[ \big|\hat{f}_k\big|^4\, \,\big|\, \sigma_Z\big] + \text{const.} \, \leq \, \text{const.}\, \cdot h_k^{-8} \\
	& \leq \, \text{const.}\, \cdot \hat \delta_n^{-4} \leq \text{const.}\, \cdot n^{2},\qquad k \in \cK_n,
	\end{align*}
	now suffices to conclude that for sufficiently large choice of the constant $\CLe$, 
	\begin{align*}
	\E \big[\big|\hat{f}_{\hat k} & - f\big|^2 \, \big|\, \sigma_Z\big]  \leq \, \mathcal{O} \Big(\big[{\cal C}_{n}(\hat{\delta}_n) \, \log n\big]^{\frac{\alpha}{\alpha+2}} \Big) +  \mathcal{O}(n^{-1})\\
	& \qquad + \text{const.}\, \cdot \,\ind\big(Z_{(j)}< -\pi/4 \text{ or } Z_{(j+1)}> \pi/4 , \quad j=1, \ldots, n-1 \big).
	\end{align*}
	The remainder of the proof is as that of Proposition \ref{P:1}.
\end{proof}

\subsection{Spacings} \label{spacings}


As $Z_j = \arctan X_j$ the density of $Z_j$ equals
$$ f_Z(z) = f_X(\tan(z)) / \cos^2 z\,, \qquad \forall z \in (-\pi/2,\pi/2)\,, $$
so that (\ref{eq:decayX}) implies
\begin{equation} \label{eq:decayZ}
C_Z \big||z|-\pi/2\big|^\beta \geq f_Z(z) \geq c_Z  \big||z|-\pi/2\big|^\beta \,, \qquad \forall z \in (-\pi/2,\pi/2)\,,
\end{equation}
for some constants $C_Z, c_Z > 0$.  

\begin{lem}\label{lem:spacings}
	If $f_X$ satisfies (\ref{eq:decayX}) and hence $f_Z$ fulfills (\ref{eq:decayZ}), then for $\kappa >1$ we have that%
	\begin{align} \nonumber
	\E \Big[\sum_{j=1}^{n-1} \big(Z_{(j+1)} - Z_{(j)}\big)^\kappa \cdot \, & \ind(\delta - \pi/2 \leq Z_{(j)}, Z_{(j+1)} \leq \pi/2 - \delta)\Big]  \\ \label{eq:boundsumspacings} & \leq 
	2 \kappa \, C_Z\,  c_Z^{-\kappa} \Gamma(\kappa) \, n (n-1)^{-\kappa} \int_\delta^{\pi/2} u^{-\beta(\kappa-1)} \di u.
	\end{align}
	Furthermore, 
	\begin{align}
	\max\Big(\E& \big[ \big(L_n(\delta)+ \pi/2\big)^2\big], \E \big[\big(R_n(\delta) - \pi/2\big)^2\big]\Big) \notag \\ & \leq 2\big(\delta + \frac{1}{c_Z\, n \, \delta^\beta}\big)^2 \, + \, \pi^2 \cdot \exp\big(- c_Z\, n (\pi/2-\delta) \delta^\beta\big) \notag \\
	& \leq 32\, \big(\delta + \frac{1}{c_Z\, n \, \delta^\beta}\big)^2, \qquad \delta \leq \pi/4,
	\end{align}	
	and for $\delta \leq \pi/4$ that
	\begin{align} \nonumber 
	 \Pb\big(Z_{(j)}< -\pi/2 + \delta \text{ or } Z_{(j+1)}> \pi/2 - \delta& , \quad j=1, \ldots, n-1 \big) \\ \label{eq:probcontains} & \leq n \, \exp\big( - c_Z\, (n-1)\,(\pi/4)^\beta \big).
	\end{align} 
 \end{lem}
\begin{proof}[Proof of Lemma \ref{lem:spacings}]
Setting
$$ Z_j^* := \begin{cases} Z_j\,, & \mbox{ if }Z_j \geq Z_k\,, \, \forall k=1,\ldots,n\,, \\
\min\{Z_k : Z_k > Z_j\}\,, & \mbox{ otherwise,} \end{cases} $$  
we deduce under (\ref{eq:decayZ}) that
\allowdisplaybreaks
\begin{align*} & \, \E \Big[ \sumdel \big(Z_{(j+1)} - Z_{(j)}\big)^\kappa\Big]\\
 = & \, \E \, \Big[\sum_{j=1}^n  \big(Z_j^* - Z_j\big)^\kappa\, \ind(\delta - \pi/2 \leq Z_{j}, Z_{j}^* \leq \pi/2 - \delta)\Big] \\
 \leq & \, n\, \E \, \Big[\E \, \big[ \big(Z_1^* - Z_1\big)^\kappa \mid Z_1\big]\, \ind(\delta - \pi/2 \leq Z_{1} \leq \pi/2 - \delta)\Big] \\
 = & \, n\,  \E \Big[\, \int_0^{(\pi/2 - Z_1)^\kappa} \Pb\big(Z_1^* > Z_1 + t^{1/\kappa} \mid Z_1\big) \di t\, \, \ind(\delta - \pi/2 \leq Z_{1} \leq \pi/2 - \delta)\Big] \\
 \leq & \, n\,  \E \,\Big[ \, \int\limits_0^{(\pi/2 - Z_1)^\kappa} \Pb\big(Z_k \not\in (Z_1,Z_1 + t^{1/\kappa}), \, \forall k\neq 1 \mid Z_1\big)  \di t \, \ind(\delta - \pi/2 \leq Z_{1} \leq \pi/2 - \delta)\Big] \\
= & \, n \, \int_{\delta-\pi/2}^{\pi/2-\delta} \, \int_0^{(\pi/2 - z)^\kappa}\,  \Big(1 - \int_{z}^{z+ t^{1/\kappa}} f_Z(x) \di x\Big)^{n-1} \, \di t\, f_Z(z) \di z \\
= & \, n \, \int_{\delta-\pi/2}^{\pi/2-\delta} \, \int_0^{\pi/2 - z}\,  \Big(1 - \int_{z}^{z+s} f_Z(x) \di x\Big)^{n-1} \kappa s^{\kappa-1} \di s\,  f_Z(z) \di z \\
 \leq & \, C_Z \, n \int_{\delta-\pi/2}^{\pi/2-\delta} \, \int_0^{\infty} \exp\big( - (n-1) c_Z ||z|-\pi/2|^\beta\, s \big) \, \kappa s^{\kappa-1} \di s \, ||z| - \pi/2|^\beta \di z \\
 = & \, \kappa \, C_Z\, c_z^{-\kappa}\, n\, (n-1)^{-\kappa}\, \int_{\delta-\pi/2}^{\pi/2-\delta} \, ||z| - \pi/2|^{-\beta(\kappa-1)} \di z\,  \int_0^{\infty} \exp( -  s) s^{\kappa-1} \di s  \\
 = & \, 2 \kappa \, C_Z\, c_Z^{-\kappa} \Gamma(\kappa) \, n (n-1)^{-\kappa}\, \int_\delta^{\pi/2} u^{-\beta(\kappa-1)} \di u\,,
\end{align*}

that is, (\ref{eq:boundsumspacings}). 
Moreover we write $Z_j^* := Z_j + \pi/2$ and $L_n^*(\delta) := L_n(\delta)+\pi/2$ so that
\allowdisplaybreaks
\begin{align*}
& \E \big[L_n^*(\delta)^2\big]  = 2\, \int_0^\pi \,z\, \Pb( L_n^*(\delta) > z)\, \di z \\
& \leq  2 \int_0^\delta \, z \, \di z + 2\, \int_\delta^\pi \, z \, \Pb\big(\min\,\big\{ Z_j^* : Z_j^* \geq  \delta\big\} \geq z \big) \, \di z\\
&  =\,  \delta^2 + 2\, \int_\delta^\pi \, z \, \Big( 1 - \int_\delta^z\, f_Z(x-\pi/2)\, \di x\Big)^n \, \di z \\ 
&  \leq \, \delta^2 + 2\,\int_\delta^\pi \, z \, \exp\Big( - n\,\int_\delta^z\, f_Z(x-\pi/2)\, \di x\Big) \, \di z\\
&  \leq \, \delta^2 + 2\,\int_\delta^{\pi/2} \, z \, \exp\big( - n \, c_Z (z - \delta)\, \delta^\beta\big) \, \di z \, + \, \pi^2 \cdot \exp\big(- n c_Z (\pi/2-\delta) \delta^\beta\big) \\ 
&   =\, \delta^2 + 2\,\int_0^{\pi/2 -\delta} \, (z+\delta) \, \exp\,\big( - n\,z \, \delta^\beta\big) \, \di z  \, + \, \pi^2 \cdot \exp\big(- n  c_Z (\pi/2-\delta) \delta^\beta\big)\\
& \leq \, \delta^2 + 2\, \frac{\delta}{c_Z\, n \, \delta^\beta}+ 2\, \frac{1}{\big(c_Z\, n \, \delta^\beta\big)^2} \, + \, \pi^2 \cdot \exp\big(- n \, c_Z\,(\pi/2-\delta) \delta^\beta\big)\\
& \leq \, 2\big(\delta + \frac{1}{c_Z\, n \, \delta^\beta}\big)^2 \, + \, \pi^2 \cdot \exp\big(- n\, c_Z\, (\pi/2-\delta) \delta^\beta\big)\,,
\end{align*}
as $\delta \downarrow 0$. The term $\E\big[\big(R_n(\delta)-\pi/2)^2\big]$ can be bounded analogously. 

Concerning (\ref{eq:probcontains}), we bound the probability that there is at most one observation in $[-\pi/2+\delta, \pi/2 - \delta]$ for $\delta \leq \pi/4$ by
\begin{align*}
&  \Pb\big(Z_{(j)}< -\pi/4 \text{ or } Z_{(j+1)}> \pi/4, \quad j=1, \ldots, n-1 \big) \\
& \, \leq \,  \, n \, \Pb\big( Z_j \in [-\pi/2, -\pi/4) \cup (\pi/4 ,\pi/2], \ j=2, \ldots, n\big)\\
& \, \leq \, n \, \Big( 1 - \int_{- \pi/4 }^{\pi/4 }\, f_Z(z)\, \di z\Big)^{n-1}\\
& \, \leq \, n \, \exp\big( - c_Z\, (n-1)\, (\pi/2)\,(\pi/4)^\beta \big),\\
\end{align*}
which implies the result. 
\end{proof}

\section*{Acknowledgements}
The authors are grateful to the editors and a referee for their thorough review and very helpful and constructive comments. H. Holzmann gratefully acknowledges financial support of the DFG, grant Ho 3260/5-1.


\begin{thebibliography}{9}

\bibitem{AB11} {\small Arellano, M. and Bonhomme, S. (2011). Identifying distributional characteristics in random coefficients panel data models. {\it Rev. Econ. Stud.} {\bf 79}, 987--1020.}

\bibitem{BFH96} {\small Beran, R. Feuerverger, A. and Hall, P. (1996). On nonparametric estimation of intercept and slope distributions in random coefficient regression. {\it Ann. Statist.} {\bf 24}, 2569--2592.}

\bibitem{BH92} {\small Beran, R. and Hall, P. (1992). Estimating coefficient distributions in random coefficient regressions. {\it Ann. Statist.} {\bf 20}, 1970--1984.}

\bibitem{BM94} {\small Beran, R. and Millar, P.W. (1994). Minimum distance estimation in random coefficient regression models. {\it Ann. Statist.} {\bf 22}, 1976--1992.}

\bibitem{BH18} {\small Breunig, C. and Hoderlein, S. (2018). Specification testing in random coefficient models. {\it Quant. Econ.} {\bf 9}, 1371--1417.}

\bibitem{DEPSH17} {\small Dunker, F., Eckle, K., Proksch, K. and Schmidt-Hieber, J. (2019). Tests for qualitative features in the random coefficients model. {\it Elect. J. Statist.} {\bf 13}, 2257--2306.}

\bibitem{GG19} {\small Gaillac, C. and Gautier, E. (2019). Adaptive estimation in the linear random coefficients model when regressors have limited variation. {\it arXiv: 1905.06584}.}

\bibitem{GH11} {\small Gautier, E. and Hoderlein, S. (2011). A triangular treatment effect model with random coefficients in the selection equation. {\it arXiv: 1109.0362}.}

\bibitem{GK13} {\small Gautier, E. and Kitamura, Y. (2013). Nonparametric estimation in random coefficients binary choice models. {\it Econometrica} {\bf 81}, 581--607.}

\bibitem{GL18} {\small Gautier, E. and Le Pennec, E. (2018). Adaptive estimation in the nonparametric random coefficients binary choice model by needlet thresholding. {\it Elect. J. Statist.} {\bf 12}, 277--320.}

\bibitem{HH68} {\small Hildreth, C. and Houck, J.P. (1968). Some estimators for a linear model with random coefficients. {\it J. Amer. Statist. Assoc.} {\bf 63}, 584--595.}

\bibitem{HHM17} {\small Hoderlein, S., Holzmann, H. and Meister, A. (2017). The triangular model with random coefficients. {\it J. Econometrics} {\bf 201}, 144--169.}

\bibitem{HKM10} {\small Hoderlein, S., Klemel{\"a}, J. and Mammen, E. (2010). Analyzing the random coefficient model nonparametrically. {\it Econometric Theo.} {\bf 26}, 804--837.}

\bibitem{JMR14} {\small Jirak, M., Meister, A. and Rei{\ss}, M. (2014). Adaptive function estimation in nonparametric regression with one-sided errors. {\it Ann. Statist.} {\bf 42}, 1970--2002.}

\bibitem{L91} {\small Lepski, O.V. (1991). Asymptotically minimax adaptive estimation. i. upper bounds. optimally adaptive estimates. {\it Teoriya Veroyatnostei i ee Primeneniya} {\bf 36}, 645--659.}

\bibitem{LS97} {\small Lepski, O.V. and Spokoiny, V.G. (1997). Optimal pointwise adaptive methods in nonparametric estimation. {\it Ann. Statist.} {\bf 25}, 2512--2546.}

\bibitem{LP17} {\small Lewbel, A. and Pendakur, K. (2017). Unobserved preference heterogeneity in demand using generalized random coefficients. {\it J. Polit. Econ.} {\bf 125}, 1100--1148.}

\bibitem{M17} {\small Masten, M.A. (2017). Random coefficients on endogenous variables in simultaneous equations models. {\it Rev. Econ. Stud.} {\bf 85}, 1193--1250.}

\bibitem{MT13} {\small Masten, M.A. and Torgovitsky, A. (2013). Instrumental variables estimation of a generalized correlated random coefficients model. {\it arXiv: 1310.6643}.}

\bibitem{M09} {\small Meister, A., {\it Deconvolution problems in nonparametric statistics}, 2009, Vol. {\bf 193}, Springer.}

\bibitem{T09} {\small Tsybakov, A.B., {\it Introduction to nonparametric estimation}, 2009, Springer.}


\end{thebibliography}
\end{document}